\documentclass[11pt, a4paper, reqno]{amsart}

\usepackage{amsmath, amssymb, amsthm}
\usepackage[T1]{fontenc}
\usepackage[utf8]{inputenc}
\usepackage[all]{xy}
\usepackage[
left=2.5cm,right=2.5cm,top=2.5cm,bottom=2.5cm,a4paper
  ]{geometry}
\usepackage{hyperref}
\usepackage{indentfirst}
\usepackage{algorithmic,algorithm}
\usepackage[last]{changes}
\usepackage{graphics,graphicx}
\definechangesauthor[color=blue]{KY}

\def\C{{\mathbb C}}

\def\F{{\mathbb F}}

\def\Q{{\mathbb Q}}
\def\R{{\mathbb R}}

\def\Z{{\mathbb Z}}

\theoremstyle{plain}
\newtheorem{theorem}{Theorem}[section]
\newtheorem{lemma}[theorem]{Lemma}
\newtheorem{proposition}[theorem]{Proposition}
\newtheorem{corollary}[theorem]{Corollary}
\newtheorem{definition}[theorem]{Definition}
\newtheorem{remark}{Remark}
\newtheorem{definition and lemma}[theorem]{Definition and Lemma}

\theoremstyle{definition} 

\newtheorem{example}[theorem]{Example}

\newtheorem{question}{Question}

\theoremstyle{remark} 

\DeclareFontFamily{U}{wncy}{}
\DeclareFontShape{U}{wncy}{m}{n}{<->wncyr10}{}
\DeclareSymbolFont{mcy}{U}{wncy}{m}{n}
\DeclareMathSymbol{\Sha}{\mathord}{mcy}{"58}

\title[Automorphism groups of absolutely simple polarized abelian varieties]{Absolutely simple polarized abelian varieties of odd Sophie Germain prime dimension over finite fields with maximal automorphism groups}
\author[WonTae Hwang]{WonTae Hwang$^*$}
\address{School of Mathematics, Korea Institute for Advanced Study, Hoegiro 85, Seoul, South Korea}
\email{hwangwon@kias.re.kr}

\author[Kyunghwan Song]{Kyunghwan Song}
\address{Ewha Institute of Mathematical Sciences, Ewha Womans University, Seoul, Republic of Korea}
\email{khsong0118@ewha.ac.kr}

\begin{document}



\maketitle

\begin{abstract}
For each Sophie Germain prime $g \geq 5,$ we construct an absolutely simple polarized abelian variety of dimension $g$ over a finite field, whose automorphism group is a cyclic group of order $4g+2$. We also provide a description on the asymptotic behavior of prime numbers that are related to our construction.  \\

\end{abstract}

\section{Introduction}
Let $k$ be a finite field, and let $X$ be an abelian variety of dimension $g$ over $k.$ We denote the endomorphism ring of $X$ over $k$ by $\textrm{End}_k(X)$. It is a free $\Z$-module of rank $\leq 4g^2.$ We also let $\textrm{End}^0_k(X)=\textrm{End}_k(X) \otimes_{\Z} \Q.$ This $\Q$-algebra $\textrm{End}_k^0(X)$ is called the endomorphism algebra of $X$ over $k.$ Then $\textrm{End}_k^0(X)$ is a finite dimensional semisimple algebra over $\Q$ with $\textrm{dim}_{\Q} \textrm{End}_k^0(X) \leq 4g^2.$ Moreover, if $X$ is $k$-simple, then $\textrm{End}_k^0(X)$ is a division algebra over $\Q$. Now, it is well known that $\textrm{End}_k (X)$ is a $\Z$-order in $\textrm{End}_k^0(X).$ The group $\textrm{Aut}_k(X)$ of the automorphisms of $X$ over $k$ is not finite, in general. But if we fix a polarization $\mathcal{L}$ on $X$, then the group $\textrm{Aut}_k(X,\mathcal{L})$ of the automorphisms of the polarized abelian variety $(X,\mathcal{L})$ is finite. Hence it might be natural to ask which finite groups can occur as such automorphism groups. Along this line, the first author\cite{Hwang(2019)} gave a classification of finite groups that can be realized as the full automorphism groups of simple polarized abelian varieties of odd prime dimension over finite fields. One of the main results of \cite{Hwang(2019)} says that the automorphism groups of simple polarized abelian varieties of odd prime dimension over finite fields are necessarily cyclic (see \cite[Theorem 4.1]{Hwang(2019)}). Furthermore, it was also shown that for a prime $g \geq 5$ with $2g+1$ being a prime, the maximal such a cyclic group is of order $4g+2.$ (Recall that such a prime $g$ is called a Sophie Germain prime.) Note that the simple abelian varieties that were constructed in the proof of \cite[Theorem 4.1-$\sharp 6$]{Hwang(2019)} become non-simple over some field extensions of the base fields.  
The goal of this paper is to consider the opposite situation, namely, to construct absolutely simple polarized abelian varieties of a Sophie Germain prime dimension over finite fields, whose automorphism group is a cyclic group of maximal order. 
In this aspect, one of our main results is the following
\begin{theorem}\label{main theorem1}
	For each Sophie Germain prime $g \geq 5,$ there exists an absolutely simple polarized abelian variety of dimension $g$ over a finite field, whose automorphism group is a cyclic group $C_{4g+2}$ of order $4g+2.$
\end{theorem}
\begin{remark}
	The case for $g=3$ was already treated in \cite[Theorem 4.1-$\sharp4$]{Hwang(2019)}.
\end{remark}

The proof of Theorem \ref{main theorem1} requires a variety of knowledge including the theories of abelian varieties over finite fields, central simple division algebras over number fields, binary quadratic forms, and the ramification theory of number fields. For more details, see Theorem \ref{main thm}. \\ 
\indent Another thing that we are interested in is to examine the distribution of the prime numbers $p$, which can be realized as the characteristic of the base field of some absolutely simple abelian varieties in our construction. In this regard, the following is our second main theorem of this paper:  
\begin{theorem}\label{main theorem2}
	For a Sophie Germain prime $g \geq 5,$ let $P_G$ be the set of primes $p$ which can be realized as the characteristic of the base field of absolutely simple polarized abelian varieties with $G:=C_{4g+2}$ as its full automorphism group. Then the (natural) density of $P_G$ is at least $\frac{1}{2 \cdot h(-8g-4)} \cdot \left(1-\frac{1}{g} \right)$.    
\end{theorem}
To prove Theorem \ref{main theorem2}, we give a careful argument with the help of the Chebotarev theorem\cite[Theorem 3.1]{Lens}. For more details, see Theorem \ref{asym thm} and Corollary \ref{asym cor}. \\


This paper is organized as follows: Section~\ref{prelim} is devoted to several facts which are related to our desired construction. Explicitly, we will recall some facts about endomorphism algebras of simple abelian varieties ($\S$\ref{end alg av}), the theorem of Tate ($\S$\ref{thm Tate sec}), Honda-Tate theory ($\S$\ref{thm Honda}), a result of Waterhouse ($\S$\ref{thm waterhouse}), and the theory of binary quadratic forms ($\S$\ref{binaryquad}). In Section~\ref{main}, we obtain the desired construction using the facts that were introduced in the previous sections. Finally, in Section \ref{asym}, we give another result regarding the distribution of prime numbers that are related to our construction. \\
\indent In the sequel, let $g \geq 5$ be a Sophie Germain prime, and let $q=p^a$ for some prime number $p$ and an integer $a \geq 1,$ unless otherwise stated. Also, let $\overline{k}$ denote an algebraic closure of a field $k.$ Finally, for an integer $n \geq 1,$ we denote the cyclic group of order $n$ (resp.\ a primitive $n$-th root of unity and resp.\ the number of classes of primitive positive definite binary quadratic forms with discriminant $-n$) by $C_n$ (resp.\ $\zeta_{n}$ and resp.\ $h(-n)$).

\section{Preliminaries}\label{prelim}
In this section, we recall some of the facts in the general theory of abelian varieties over a field. Our main references are \cite{2} and \cite{8}.

\subsection{Endomorphism algebras of simple abelian varieties of odd prime dimension over finite fields}\label{end alg av}
In this section, we review all the possible endomorphism algebras of simple abelian varieties of odd prime dimension over finite fields. \\
\indent Let $X$ be a simple abelian variety of dimension $g$ over a finite field $k.$ Then it is well known that $\textrm{End}_k^0(X)$ is a division algebra over $\Q$ with $2g \leq \textrm{dim}_{\Q} \textrm{End}_k^0(X) < (2g)^2.$ Before giving our first result, we also recall Albert's classification. We choose a polarization $\lambda : X \rightarrow \widehat{X}$ where $\widehat{X}$ denotes the dual abelian variety of $X$. Using the polarization $\lambda,$ we can define an involution, called the \emph{Rosati involution}, $^{\vee}$ on $\textrm{End}_k^0(X).$ (For a more detailed discussion about the Rosati involution, see \cite[\S20]{8}.) In this way, to the pair $(X,\lambda)$ we associate the pair $(D, ^{\vee})$ with $D=\textrm{End}_k^0(X)$ and $^{\vee}$, the Rosati involution on $D$. We know that $D$ is a simple division algebra over $\Q$ of finite dimension and that $^{\vee}$ is a positive involution. Let $K$ be the center of $D$ so that $D$ is a central simple $K$-algebra, and let $K_0 = \{ x \in K~|~x^{\vee} = x \}$ be the subfield of symmetric elements in $K.$ By a theorem of Albert, $D$ (together with $^\vee$) is of one of the following four types: \\
\indent (i) Type I: $K_0 = K=D$ is a totally real field. \\
\indent (ii) Type II: $K_0 = K$ is a totally real field, and $D$ is a quaternion algebra over $K$ with $D\otimes_{K,\sigma} \R \cong M_2(\R)$ for every embedding $\sigma : K \hookrightarrow \R.$  \\
\indent (iii) Type III: $K_0 = K$ is a totally real field, and $D$ is a quaternion algebra over $K$ with $D\otimes_{K,\sigma} \R \cong \mathbb{H}$ for every embedding $\sigma : K \hookrightarrow \R$ (where $\mathbb{H}$ is the Hamiltonian quaternion algebra over $\R$). \\
\indent (iv) Type IV: $K_0 $ is a totally real field, $K$ is a totally imaginary quadratic field extension of $K_0$, and $D$ is a central simple algebra over $K$. \\
\indent Keeping the notations as above, we let
\begin{equation*}
e_0= [K_0 :\Q],~~~e=[K:\Q],~~~\textrm{and}~~~d=[D:K]^{\frac{1}{2}}.
\end{equation*}

As our last preliminary fact of this section, we impose some numerical restrictions on those values $e_0, e,$ and $d$ in the next table, following \cite[\S21]{8}.
\begin{center}
	\begin{tabular}{c c c}
		\hline
		& $\textrm{char}(k)=0$ & $\textrm{char}(k)=p>0$ \\
		\hline
		$\textrm{Type I}$ & $e|g$ & $e|g$  \\
		\hline
		$\textrm{Type II}$ & $2e|g$  & $2e|g$ \\
		\hline
		$\textrm{Type III}$ & $2e|g$ & $e|g$ \\
		\hline
		$\textrm{Type IV}$ & $e_0 d^2 |g$ &$ e_0 d |g $ \\
		\hline
	\end{tabular}
	\vskip 4pt
	\textnormal{Table 1 } \\
	\textnormal{Numerical restrictions on endomorphism algebras.}
\end{center}

Now, we are ready to introduce the following
\begin{lemma}\label{poss end alg}
	Let $X$ be a simple abelian variety of dimension $g$ over a finite field $k=\F_q$, and let $\lambda : X \rightarrow \widehat{X}$ be a polarization. Then $D:=\textrm{End}_k^0(X)$ (together with the Rosati involution $^\vee$ corresponding to $\lambda$) is of one of the following two types: \\
	(1) $D$ is a totally imaginary quadratic extension field of a totally real field of degree $g$; \\
	(2) $D$ is a central simple division algebra of degree $g$ over an imaginary quadratic field.
\end{lemma}
\begin{proof}
	For a proof, see \cite[Lemma 2.1]{Hwang(2019)}.
\end{proof}

We also recall the following fact on the possible endomorphism algebras for the case when the base field is algebraically closed of positive characteristic:
\begin{remark}\label{oort rmk}
	Let $K$ be an algebraically closed field with $\textrm{char}(K)=p>0.$ Let $X$ be a simple abelian variety of dimension $g$ over $K$, and let $D=\textrm{End}_K^0(X)$. Then $D$ is of one of the following types (see \cite[$\S7$]{10}): \\
	(i) $D=\Q$;\\
	(ii) $D$ is a totally real field of degree $g$;\\
	(iii) $D=D_{p,\infty}$ if $g \geq 5$, where $D_{p,\infty}$ is the unique quaternion algebra over $\Q$ that is ramified at $p$ and $\infty$, and split at all other primes; \\
	(iv) $D$ is an imaginary quadratic field;\\
	(v) $D$ is a totally imaginary quadratic extension field of a totally real field of degree $g$; \\
	(vi) $D$ is a central simple division algebra of degree $g$ over an imaginary quadratic field and the $p$-rank of $X$ is $0.$
\end{remark}

\subsection{The theorem of Tate}\label{thm Tate sec}
In this section, we recall an important theorem of Tate, and give some interesting consequences of it. \\
\indent Let $k$ be a field and let $l$ be a prime number with $l \ne \textrm{char}(k)$. If $X$ is an abelian variety of dimension $g$ over $k,$ then we can introduce the Tate $l$-module $T_l X$ and the corresponding $\Q_l$-vector space $V_l X :=T_l X \otimes_{\Z_l} \Q_l.$ It is well known that $T_l X$ is a free $\Z_l$-module of rank $2g$ and $V_l X$ is a $2g$-dimensional $\Q_l$-vector space. In \cite{11}, Tate showed the following important result:

\begin{theorem}\label{thm Tate}
	Let $k$ be a finite field and let $\Gamma = \textrm{Gal}(\overline{k}/k).$ If $l$ is a prime number with $l \ne \textrm{char}(k),$ then we have: \\
	(a) For any abelian variety $X$ over $k,$ the representation
	\begin{equation*}
	\rho_l =\rho_{l,X} : \Gamma \rightarrow \textrm{GL}(V_l X)
	\end{equation*}
	is semisimple. \\
	(b) For any two abelian varieties $X$ and $Y$ over $k,$ the map
	\begin{equation*}
	\Z_l \otimes_{\Z} \textrm{Hom}_k(X,Y) \rightarrow \textrm{Hom}_{\Gamma}(T_l X, T_l Y)
	\end{equation*}
	is an isomorphism.
\end{theorem}

Now, we recall that an abelian variety $X$ over a (finite) field $k$ is called \emph{elementary} if $X$ is $k$-isogenous to a power of a simple abelian variety over $k.$ Then, as an interesting consequence of Theorem \ref{thm Tate}, we have the following
\begin{corollary}\label{cor TateEnd0}
	Let $X$ be an abelian variety of dimension $g$ over a finite field $k.$ Then we have:\\
	(a) The center $Z$ of $\textrm{End}_k^0(X)$ is the subalgebra $\Q[\pi_X].$ In particular, $X$ is elementary if and only if $\Q[\pi_X]=\Q(\pi_X)$ is a field, and this occurs if and only if $f_X$ is a power of an irreducible polynomial in $\Q[t]$ where $f_X$ denotes the characteristic polynomial of $\pi_X.$ \\
	(b) Suppose that $X$ is elementary. Let $h=f_{\Q}^{\pi_X}$ be the minimal polynomial of $\pi_X$ over $\Q$. Further, let $d=[\textrm{End}_k^0(X):\Q(\pi_X)]^{\frac{1}{2}}$ and $e=[\Q(\pi_X):\Q].$ Then $de =2g$ and $f_X = h^d.$ \\
	(c) We have $2g \leq \textrm{dim}_{\Q} \textrm{End}^0_k (X) \leq (2g)^2$ and $X$ is of CM-type. \\
	(d) The following conditions are equivalent: \\
	\indent (d-1) $\textrm{dim}_{\Q} \textrm{End}_k^0(X)=2g$; \\
	\indent (d-2) $\textrm{End}_k^0(X)=\Q[\pi_X]$; \\
	\indent (d-3) $\textrm{End}_k^0(X)$ is commutative; \\
	\indent (d-4) $f_X$ has no multiple root. \\
	(e) The following conditions are equivalent: \\
	\indent (e-1) $\textrm{dim}_{\Q} \textrm{End}_k^0(X)=(2g)^2$; \\
	\indent (e-2) $\Q[\pi_X]=\Q$; \\
	\indent (e-3) $f_X$ is a power of a linear polynomial; \\
	\indent (e-4) $\textrm{End}^0_k(X) \cong M_g(D_{p,\infty})$ where $D_{p,\infty}$ is the unique quaternion algebra over $\Q$ that is ramified at $p$ and $\infty$, and split at all other primes; \\
	\indent (e-5) $X$ is supersingular with $\textrm{End}_k(X) = \textrm{End}_{\overline{k}}(X_{\overline{k}})$ where $X_{\overline{k}}=X \times_k \overline{k}$; \\
	\indent (e-6) $X$ is isogenous to $E^g$ for a supersingular elliptic curve $E$ over $k$ all of whose endomorphisms are defined over $k.$
\end{corollary}
\begin{proof}
	For a proof, see \cite[Theorem 2]{11}.
\end{proof}

For a precise description of the structure of the endomorphism algebra of a simple abelian variety $X$, viewed as a simple algebra over its center $\Q[\pi_X]$, we record the following useful result:
\begin{proposition}\label{local inv}
	Let $X$ be a simple abelian variety over a finite field $k=\F_q.$ Let $K=\Q[\pi_X].$ Then we have: \\
	(a) If $\nu$ is a place of $K$, then the local invariant of $\textrm{End}_k^0(X)$ in the Brauer group $\textrm{Br}(K_{\nu})$ is given by
	\begin{equation*}
	\textrm{inv}_{\nu}(\textrm{End}_k^0(X))=\begin{cases} 0 & \mbox{if $\nu$ is a finite place not above $p$}; \\ \frac{\textrm{ord}_{\nu}(\pi_X)}{\textrm{ord}_{\nu}(q)} \cdot [K_{\nu}:\Q_p] & \mbox{if $\nu$ is a place above $p$}; \\ \frac{1}{2} & \mbox{if $\nu$ is a real place of $K$}; \\ 0 & \mbox{if $\nu$ is a complex place of $K$}. \end{cases}
	\end{equation*}
	(b) If $d$ is the degree of the division algebra $D:=\textrm{End}_k^0(X)$ over its center $K$ (so that $d=[D:K]^{\frac{1}{2}}$ and $f_X = (f_{\Q}^{\pi_X})^d$), then $d$ is the least common denominator of the local invariants $\textrm{inv}_{\nu}(D).$
\end{proposition}
\begin{proof}
	(a) For a proof, see \cite[Corollary 16.30]{2}. \\
	(b) For a proof, see \cite[Corollary 16.32]{2}.
\end{proof}

\subsection{Abelian varieties up to isogeny and Weil numbers: Honda-Tate theory}\label{thm Honda}
In this section, we recall an important theorem of Honda and Tate. To achieve our goal, we first give the following
\begin{definition}\label{qWeil Def}
	(a) A \emph{$q$-Weil number} is an algebraic integer $\pi$ such that $| \iota(\pi) | = \sqrt{q}$ for all embeddings $\iota : \Q[\pi] \hookrightarrow \C.$ \\
	(b) Two $q$-Weil numbers $\pi$ and $\pi^{\prime}$ are said to be \emph{conjugate} if they have the same minimal polynomial over $\Q,$ or equivalently, there is an isomorphism $\Q[\pi] \rightarrow \Q[\pi^{\prime}]$ sending $\pi$ to $\pi^{\prime}.$
\end{definition}

Regarding $q$-Weil numbers, the following facts are well-known:
\begin{remark}\label{qWeil rem}
	Let $X$ and $Y$ be simple abelian varieties over a finite field $k=\F_q.$ Then we have \\
	(1) The Frobenius endomorphism $\pi_X$ is a $q$-Weil number. \\
	(2) $X$ and $Y$ are $k$-isogenous if and only if $\pi_X$ and $\pi_Y$ are conjugate.
\end{remark}

Now, we introduce our main result of this section:

\begin{theorem}\label{thm HondaTata}
	For every $q$-Weil number $\pi$, there exists a simple abelian variety $X$ over $\F_q$ such that $\pi_X$ is conjugate to $\pi$. Moreover, we have a bijection between the set of isogeny classes of simple abelian varieties over $\F_q$ and the set of conjugacy classes of $q$-Weil numbers given by $X \mapsto \pi_X$.
\end{theorem}
The inverse of the map $X \mapsto \pi_X$ associates to a $q$-Weil number $\pi$ a simple abelian variety $X$ over $\F_q$ such that $f_X$ is a power of the minimal polynomial $f_{\Q}^{\pi}$ of $\pi$ over $\Q.$
\begin{proof}
	For a proof, see \cite[Main Theorem]{4} or \cite[\S16.5]{2}.
\end{proof}

\subsection{Isomorphism classes contained in an isogeny class}\label{thm waterhouse}
In this section, we will give a useful result of Waterhouse~\cite{12}. Throughout this section, let $k=\F_q.$ \\
\indent Let $X$ be an abelian variety over $k.$ Then $\textrm{End}_k(X)$ is a $\Z$-order in $\textrm{End}_k^0(X)$ containing $\pi_X$ and $q/\pi_X.$ If a ring is the endomorphism ring of an abelian variety, then we may consider a left ideal of the ring, and give the following
\begin{definition}\label{Roccur}
	Let $X$ be an abelian variety over $k$ with $R:=\textrm{End}_k(X),$ and let $I$ be a left ideal of $R$, which contains an isogeny. \\
	(a) We define $H(X,I)$ to be the intersection of the kernels of all elements of $I$. This is a finite subgroup scheme of $X.$ \\
	(b) We define $X_I$ to be the quotient of $X$ by $H(X,I)$ i.e.\ $X_I=X/H(X,I).$ This is an abelian variety over $k$ that is $k$-isogenous to $X.$
\end{definition}

Now, we introduce our main result of this section, which plays an important role:
\begin{proposition}\label{endposs}
	Let $X$ be an abelian variety over $k$ with $R:=\textrm{End}_k(X)$, and let $I$ be a left ideal of $R$, which contains an isogeny. Also, let $D=\textrm{End}_k^0(X).$ Then we have: \\
	(a) $\textrm{End}_k(X_I)$ contains $O_r (I):=\{x \in \textrm{End}_k^0(X)~|~I x \subseteq I\}$, the right order of $I,$ and equals it if $I$ is a kernel ideal\footnote{Recall that $I$ is a kernel ideal of $R$ if $I=\{\alpha \in R~|~\alpha H(X,I)=0 \}$.}. \\
	(b) Every maximal order in $D$ occurs as the endomorphism ring of an abelian variety in the isogeny class of $X.$
\end{proposition}
\begin{proof}
	(a) For a proof, see \cite[Lemma 16.56]{2} or \cite[Proposition 3.9]{12}. \\
	(b) For a proof, see \cite[Theorem 3.13]{12}.
\end{proof}

\subsection{Representability of integers by binary quadratic forms}\label{binaryquad}
In this section, we recall one interesting fact on the representability of integers by binary quadratic forms, which will be used later in this paper.  
\begin{proposition}\label{prop_quad_1}
	An integer $m$ is properly represented by some binary quadratic form $aX^2 + bXY + cY^2$ with $\gcd(a,b,c) = 1$ of
	discriminant $d := b^2 - 4ac$ if and only if $d$ is a square mod $4m$.
\end{proposition}
\begin{proof}
	For a proof, see \cite[Proposition 12.3.1]{Granv(2019)}.
\end{proof}
In Section \ref{main} below, we will consider the case when $n=p$ is a prime and $a=1,b=0, c=2g+1$ for a Sophie Germain prime $g \geq 5.$ For this reason, we record the following fact:

\begin{theorem}\label{prime rep thm0}
	Let $n>0$ be an integer. Then there exists a monic irreducible polynomial $f_n(t) \in \Z[t]$ of degree $h(-4n)$ such that if an odd prime $p$ divides neither $n$ nor the discriminant of $f_n(t)$, then $p=x^2+ny^2$ for some $x,y \in \Z$ if and only if $\left( \frac{-n}{p} \right)=1$ and $f_n(t) \equiv 0~(\textrm{mod}~p)$ has an integer solution.
\end{theorem}
\begin{proof}
	For a proof, see \cite[Theorem 9.2]{Cox(1989)}.
\end{proof}

Theorem \ref{prime rep thm0} can also be stated in terms of the splitting behavior of primes in a Galois extension of an imaginary quadratic field as in the next theorem, which plays an important role in Section \ref{asym}. 
\begin{theorem}\label{prime rep thm1}
	Let $n >0$ be an integer and let $L$ be the ring class field of the order $\Z[\sqrt{-n}]$ in $K:=\Q(\sqrt{-n}).$ If $p$ is an odd prime not dividing $n,$ then $p=x^2 + ny^2$ for some $x,y \in \Z$ if and only if $p$ splits completely in $L.$
\end{theorem}
\begin{proof}
	For a proof, see \cite[Theorem 9.4]{Cox(1989)}.
\end{proof}

\section{Main Result}\label{main}
In this section, we will construct an absolutely simple polarized abelian variety of a Sophie Germain dimension over a finite field whose automorphism group is a cyclic group of maximal possible order to obtain one of our main results. (We recall that a simple abelian variety over a field $k$ is \emph{absolutely simple (or geometrically simple)} if it is simple over $\overline{k}.$) \\
\indent For such a construction, we first need the following

\begin{lemma}\label{lem_p_1_2g+1}
	For each Sophie Germain prime $g \geq 5,$ there exists a prime $p$ such that $p = x^2 + (2g+1)y^2$ for some $x,y \in \mathbb{Z}$ with $p \ne 2g+1$ and $p \not\equiv 1\pmod{2g+1}$.  
\end{lemma}
\begin{proof}
	Let $K=\Q(\sqrt{-(2g+1)})$, $\mathcal{O}=\Z[\sqrt{-(2g+1)}]$ an order in $K$ with the discriminant $-4(2g+1),$ and let $L$ be the ring class field of $\mathcal{O}.$ 
	We claim first that $\Q(\zeta_{2g+1})$ is not contained in $L.$ Indeed, note that the conductor of $\mathcal{O}$ equals $2$, and hence, it follows that the ramification index of $2g+1$ in $L$ equals to $2,$ while the ramification index of $2g+1$ in $\Q(\zeta_{2g+1})$ equals to $2g >2$ because $2g+1$ is totally ramified in $\Q(\zeta_{2g+1}).$ Thus we can see that $\Q(\zeta_{2g+1})$ is not contained in $L.$ In particular, there exists a $\sigma \in \textrm{Gal}(L(\zeta_{2g+1})/\Q)$ which maps to the identity in $\textrm{Gal}(L/\Q)$ but not in $\textrm{Gal}(\Q(\zeta_{2g+1})/\Q).$ By Chebotarev density theorem, this implies that there exists a prime $p \in \Q$ such that $p$ splits completely in $L$ but not in $\Q(\zeta_{2g+1}).$ Then it follows that $p$ is a prime of the form $p=x^2 + (2g+1)y^2$ for some $x,y \in \Z$ with $p \ne 2g+1$ and $p \not \equiv 1 ~(\textrm{mod}~2g+1).$ \\
	\indent This completes the proof.
\end{proof}

\indent Now, we are ready to introduce our main result of this section:
\begin{theorem}\label{main thm}
	Let $g \geq 5$ be a Sophie Germain prime. Then there exists a finite field $k$ and an absolutely simple abelian variety $X$ of dimension $g$ over $k$ with a polarization $\mathcal{L}$ such that $\textrm{Aut}_k(X,\mathcal{L}) = C_{4g+2}$\footnote{Note that in view of \cite[Corollary 3.6]{Hwang(2019)}, $C_{4g+2}$ is a maximal such group.}.
\end{theorem}

\begin{proof}
	Let $p$ be a prime with the following two properties (P1) and (P2):\\
	\indent (P1) $p$ can be written as $p = x^2 + (2g+1)y^2$ for some $x,y \in \mathbb{Z}$ with $p \ne 2g+1$;\\
	\indent (P2) $p \not\equiv 1\pmod{2g+1}$.\\ 
	(The existence of such a prime $p$ is guaranteed by Lemma \ref{lem_p_1_2g+1}.)\\
	\indent By (P1), there exist $a,s \in \mathbb{Z}$ with $(p,a)=1$ such that the triple $(a,p,s)$ is a solution of the equation $X^2 - 4Y = -(2g+1)Z^2$. Then let $\pi$ be a zero of the quadratic polynomial $t^2 + ap^{\frac{g-1}{2}} t + p^g \in \mathbb{Z}[t]$  so that $\pi$ is a $p^g$-Weil number. By Theorem \ref{thm HondaTata}, there exists a simple abelian variety $X$ of dimension $r$ over $k:=\F_{p^g}$ such that $\pi_X$ is conjugate to $\pi,$ and, in fact, we can take $r=g$ by \cite[Proposition 2.5]{Maisner(2002)}\footnote{Alternatively, we can compute the local invariants directly to see that they are $(\frac{g-1}{2})/g, (\frac{g+1}{2})/g$ at the two places of $K$ above $p$, and $0$ at other places of $K$, and use Proposition \ref{local inv}-(b).}. Hence it follows that we have $\Q(\pi_X)=\Q(\sqrt{-(2g+1)})$\footnote{This essentially follows from the fact that $(a,p,s)$ is a solution of $X^2 - 4Y = -(2g+1)Z^2$. } and $D:=\textrm{End}_k^0(X)$ is a central simple division algebra of degree $g$ over $K:=\Q(\sqrt{-(2g+1)})$ by Lemma \ref{poss end alg}. Next, we claim that there is an embedding of $L:=\Q(\zeta_{4g+2})$ over $K$ into $D.$ Indeed, let $\mathfrak{p}$ be a prime of $K$ lying over $p$ and let $\mathfrak{P}$ be a prime of $L$ lying over $\mathfrak{p}.$ In particular, $\mathfrak{P}$ lies over $p.$ Since $p$ splits into a product of two primes in both $K$ and $L,$ 
	we have $[L_{\mathfrak{P}}:K_{\mathfrak{p}}]=g.$ On the other hand, it follows from Proposition \ref{local inv}-(a) that $D_{\mathfrak{p}}:= D \otimes_{K} K_{\mathfrak{p}}$ is a central simple division algebra of degree $g$ over $K_{\mathfrak{p}}$, and hence, we can see that $L$ embeds over $K$ into $D$ by the first theorem in \cite[p. 407]{7}, as desired. Now, it is well known that such $D$ has a maximal $\Z$-order $\mathcal{O}$ containing the ring of integers $\Z[\zeta_{4g+2}]$ of $\Q(\zeta_{4g+2}).$ Since $\mathcal{O}$ is a maximal $\Z$-order in $D,$ there exists a simple abelian variety $X^{\prime}$ over $k$ such that $X^{\prime}$ is $k$-isogenous to $X$ and $\textrm{End}_k(X^{\prime})=\mathcal{O}$ by Proposition \ref{endposs}-(b). Note also that 
	\begin{equation*}
	C_{4g+2} \cong \left \langle \zeta_{4g+2} \right \rangle \leq \Z [\zeta_{4g+2}]^{\times} =\Z^{g-1} \times \left \langle \zeta_{4g+2} \right \rangle \leq \mathcal{O}^{\times}=\textrm{Aut}_k(X^{\prime}).
	\end{equation*} 
	Now, let $\mathcal{L}$ be an ample line bundle on $X^{\prime}$, and put
	\begin{equation*}
	\mathcal{L}^{\prime}:=\bigotimes_{f \in \left \langle \zeta_{4g+2} \right \rangle} f^* \mathcal{L}.
	\end{equation*}
	Then $\mathcal{L}^{\prime}$ is an ample line bundle on $X^{\prime}$ that is preserved under the action of $\left \langle \zeta_{4g+2} \right \rangle$ so that $\left \langle \zeta_{4g+2} \right \rangle \leq \textrm{Aut}_k(X^{\prime},\mathcal{L}^{\prime}).$ Also, by \cite[Corollary 3.6]{Hwang(2019)}, we can see that $\left \langle \zeta_{4g+2} \right \rangle$ is a maximal finite subgroup of the multiplicative subgroup of $\textrm{End}_k^0(X^{\prime})=D$. Then since $\textrm{Aut}_k (X^{\prime},\mathcal{L}^{\prime})$ is a finite subgroup of $D^{\times},$ it follows that
	\begin{equation*}
	C_{4g+2} \cong \left \langle \zeta_{4g+2} \right \rangle =\textrm{Aut}_k (X^{\prime}, \mathcal{L}^{\prime}).
	\end{equation*}
	Finally, since $\pi_{X^{\prime}}$ is conjugate to $\pi,$ we can conclude that $X^{\prime}$ is absolutely simple by \cite[Proposition 3]{5}. \\
	\indent This completes the proof.
\end{proof}

\begin{remark}\label{m<g-1/2 rmk}
	More generally, we can also consider the following conditions on the existence of solutions of certain Diophantine equations: \\
	\indent (P1)$_m^{\prime}$ For each $1 \leq m \leq \frac{g-1}{2},$ there exist $a,s \in \mathbb{Z}$ with $(p,a)=1$ such that the triple $(a,p,s)$ is a solution of the equation
	\begin{equation}\label{eqn 1}
	X^2 - 4Y^{g-2m} = -(2g+1)Z^2.
	\end{equation}
	For instance, we have seen in the proof of Theorem \ref{main thm} that the condition (P1) implies the condition (P1)$_{\frac{g-1}{2}}^{\prime}$. These conditions will appear in Section \ref{asym}, again.
\end{remark}

\begin{remark}
	In view of \cite[Proposition 3.5 and Theorem 3.6]{Gon(1998)}, the absolutely simple abelian variety $X$ and $X^{\prime}$ in the proof of Theorem \ref{main thm} have $p$-rank $0.$ It might be interesting to see whether there is an example with $p$-rank being equal to $g$ (i.e.\ an example of $X$ and $X^{\prime}$ being absolutely simple ordinary abelian varieties). In this case, we note that $\textrm{End}_k^0(X)=\textrm{End}_k^0(X^{\prime})$ needs to be a CM-field of degree $2g$ by Lemma \ref{poss end alg}. (In fact, this is related to Question \ref{ques 1} below.) 
\end{remark}

We conclude this section by introducing some tables that are related to Theorem \ref{main thm}. Recall that for a given Sophie Germain prime $g \geq 5,$ the triple $(a,p,s)$ is a solution of the equation $X^2 -4Y = -(2g+1) Z^2.$ Then the following algorithm can be used to provide a quadruple $(g,p,a,s)$ for each fixed $g$ with the smallest value of the prime $p$, that is essentially obtained from the properties (P1) and (P2): 

\begin{algorithm} [!hbt]\label{Alg_1}
	\caption{~Provide a quadruple $(g,p,a,s)$ with the smallest value of the prime $p$ using (P1) and (P2) for a fixed Sophie Germain prime $g \geq 5$ and $m = \frac{g-1}{2}$ in Eqn. (3.1).}
	\begin{algorithmic}[1]
		\STATE Choose a Sophie Germain prime $g \geq 5$
		\STATE \textbf{while} $i \geq 1$ \textbf{do}
		\STATE ~~\quad $p = (i$-th prime number$)$
		\STATE ~~\quad $s_0 = 1$
		\STATE ~~\quad \textbf{while} $s_0 < \sqrt{\frac{p}{2g+1}}$ \textbf{do}
		\STATE ~~\quad\quad \textbf{if} $p - (2g+1)s_0^2$ is square \textbf{do}
		\STATE ~~\quad\quad\quad \textbf{if} $p \not= 2g+1$ and $p \not\equiv 1\pmod{2g+1}$ \textbf{do}
		\STATE ~~\quad\quad\quad\quad $a_0=\sqrt{p-(2g+1)s_0^2}$ 
		\STATE ~~\quad\quad\quad\quad $a = 2a_0, s = 2s_0$ 
		\STATE ~~\quad\quad\quad\quad print $(g,p,a,s)$
		\STATE ~~\quad\quad\quad \textbf{end if}
		\STATE ~~\quad\quad \textbf{end if}
		\STATE ~~\quad \textbf{end while}
		\STATE \textbf{end while}
	\end{algorithmic}
\end{algorithm}

Using Algorithm 1 directly, and by modifying Algorithm 1 slightly, we can obtain the following two tables:
\begin{center}
	\begin{tabular}{c c c c || c c c c}
		\hline
		$g$ & $p$ & $a$ & $s$ & $g$ & $p$ & $a$ & $s$\\
		\hline
		$5$ & $47$ & $12$ & $2$ & $191$ & $419$ & $12$ & $2$ \\
		\hline
		$11$ & $59$ & $12$ & $2$ & $233$ & $503$ & $12$ & $2$ \\
		\hline
		$23$ & $83$ & $12$ & $2$ & $239$ & $1997$ & $18$ & $4$ \\
		\hline
		$29$ & $317$ & $18$ & $4$ & $251$ & $647$ & $24$ & $2$ \\
		\hline
		$41$ & $227$ & $24$ & $2$ & $281$ & $599$ & $12$ & $2$ \\
		\hline
		$53$ & $251$ & $24$ & $2$ & $293$ & $911$ & $36$ & $2$ \\
		\hline
		$83$ & $311$ & $24$ & $2$ & $359$ & $863$ & $24$ & $2$ \\
		\hline
		$89$ & $503$ & $36$ & $2$ & $419$ & $983$ & $24$ & $2$ \\
		\hline
		$113$ & $263$ & $12$ & $2$ & $431$ & $1187$ & $36$ & $2$ \\
		\hline
		$131$ & $587$ & $36$ & $2$ & $443$ & $1031$ & $24$ & $2$ \\
		\hline
		$173$ & $383$ & $12$ & $2$ & $491$ & $1019$ & $12$ & $2$ \\
		\hline
		$179$ & $503$ & $24$ & $2$ & $509$ & $1163$ & $24$ & $2$ \\
		\hline
	\end{tabular}
	\vskip 4pt
	\textnormal{Table 2 } \\
	Quadruples \textnormal{$(g,p,a,s)$ with the smallest prime $p$ using Algorithm $1$ for each $g \leq 509$.}
\end{center}

\begin{center}
	\begin{tabular}{c c c || c c c}
		\hline
		$p$ & $a$ & $s$ & $p$ & $a$ & $s$\\
		\hline
		$59$ & $12$ & $2$ & $593$ & $30$ & $8$ \\
		\hline
		$101$ & $6$ & $4$ & $607$ & $40$ & $6$ \\
		\hline
		$167$ & $24$ & $2$ & $719$ & $24$ & $10$ \\
		\hline
		$173$ & $18$ & $4$ & $809$ & $42$ & $8$ \\
		\hline
		$211$ & $4$ & $6$ & $821$ & $54$ & $4$ \\
		\hline
		$223$ & $8$ & $6$ & $853$ & $10$ & $12$ \\
		\hline
		$271$ & $16$ & $6$ & $877$ & $14$ & $12$ \\
		\hline
		$307$ & $20$ & $6$ & $883$ & $52$ & $6$ \\
		\hline
		$317$ & $30$ & $4$ & $991$ & $56$ & $6$ \\
		\hline
		$347$ & $36$ & $2$ & $997$ & $26$ & $12$ \\
		\hline
		$449$ & $18$ & $8$ & $1097$ & $54$ & $8$ \\
		\hline
		$463$ & $32$ & $6$ & $1117$ & $34$ & $12$ \\
		\hline
	\end{tabular}
	\vskip 4pt
	\textnormal{Table 3 } \\
	Triples \textnormal{$(p,a,s)$ with fixed $g = 11$ using Algorithm 1.}
\end{center}

\section{Asymptotic behavior}\label{asym}
Let $g \geq 5$ be a (fixed) Sophie Germain prime, and let $G=C_{4g+2}.$ Let $P$ be the set of all prime numbers and define $P_G$ to be the set of all prime numbers $p$ with the property that there exists an absolutely simple abelian variety $X$ of dimension $g$ over a finite field $k$ such that $\textrm{Aut}_k (X,\mathcal{L})=G$ for some polarization $\mathcal{L}$ on $X.$ Then the following remark gives a partial description on the set $P_G$, based on the argument in the proof of Theorem \ref{main thm}.
\begin{remark}\label{PG rmk}
	For a prime $p \in P,$ we recall the following three conditions:\\
	(P1) $p$ can be written as $p = x^2 + (2g+1)y^2$ for some $x,y \in \mathbb{Z}$ with $p \ne 2g+1$;\\
	(P1)$_m^{\prime}$ For each $1 \leq m \leq \frac{g-1}{2},$ there exist $a,s \in \mathbb{Z}$ with $(p,a)=1$ such that the triple $(a,p,s)$ is a solution of the equation
	\begin{equation*}
	X^2 - 4Y^{g-2m} = -(2g+1)Z^2;
	\end{equation*}
	(P2) $p \not \equiv 1~(\textrm{mod}~2g+1).$ \\
	\indent We have seen that if a prime $p$ satisfies both (P1) and (P2), then $p \in P_G$. (Here, note again that we have assumed $m=\frac{g-1}{2}$ in the equation of the condition (P1)$_m^{\prime}$.)
\end{remark}
Now, it might be interesting to consider the following
\begin{question}\label{ques 1}
	Find the asymptotic behavior of the counting function $f_G(x):=\frac{\left|\{p \leq x~|~ p \in P_G \}\right|}{\left|\{p \leq x  ~|~ p \in P\}\right|}$ as $x \rightarrow \infty.$
\end{question}
It seems to be somewhat difficult to answer Question \ref{ques 1} completely at the moment. Hence in this section, we give an answer for the following slightly easier version of Question \ref{ques 1} in view of Remark \ref{PG rmk}: for a Sophie Germain prime $g \geq 5,$ let $P_g$ be the set of all primes satisfying both of the conditions (P1) and (P2) so that $P_g \subseteq P_G$. 
\begin{question}\label{easier ques}
	Find the asymptotic behavior of the counting function $f_g(x):=\frac{\left|\{p \leq x~|~ p \in  P_g \}\right|}{\left|\{p \leq x ~|~p \in P\}\right|}$ as $x \rightarrow \infty.$
\end{question}


To answer this question, it is useful to have the following lemma: in the sequel, let $K=\Q(\sqrt{-(2g+1)})$, $\mathcal{O}=\Z[\sqrt{-(2g+1)}]$ an order in $K$, and let $L$ be the ring class field of $\mathcal{O}$ for a Sophie Germain prime $g \geq 5.$ 
\begin{lemma}\label{asym lem}
	Let $g \geq 5$ be a Sophie Germain prime. Suppose that a prime $p$ satisfies the property (P1). Then $p$ splits completely in $L(\zeta_{2g+1})$ if and only if $p \equiv 1 ~(\textrm{mod}~2g+1)$.
\end{lemma}
\begin{proof}
	Suppose first that $p$ splits completely in $L(\zeta_{2g+1}).$ Then it follows that $p$ splits completely in $\Q(\zeta_{2g+1}),$ which, in turn, implies that we have $p \equiv 1~(\textrm{mod}~2g+1).$ Conversely, suppose that $p \equiv 1~(\textrm{mod}~2g+1)$ so that $p$ splits completely in $\Q(\zeta_{2g+1}).$ By the property (P1) and Theorem \ref{prime rep thm1}, we also know that $p$ splits completely in $L.$ Hence we can see that $p$ splits completely in $L(\zeta_{2g+1})$ by \cite[Lemma 3.3.30]{Cohen(2008)}. \\
	\indent This completes the proof. 
\end{proof}

Now, the following theorem gives an answer for Question \ref{easier ques}.
\begin{theorem}\label{asym thm}
	For a fixed Sophie Germain prime $g \geq 5$, the counting function $f_g(x)$ tends to $\frac{1}{2 \cdot h(-8g-4)} \cdot \left(1-\frac{1}{g} \right)$ as $x \rightarrow \infty.$
\end{theorem}
\begin{proof}
	For a (fixed) arbitrary $x,$ consider the following three sets: \\
	\indent $S(x)=\{p \leq x~|~ p~\textrm{satisfies (P1)}\},$ $S^{\prime}(x) =\{p \leq x~|~p \in P_g \}$, and \\
	\indent $S^{\prime\prime}(x) = \{p \leq x~|~p~\textrm{satisfies (P1) and}~p \equiv 1 ~(\textrm{mod}~2g+1)\}.$ \\
	Clearly, we have $S(x)=S^{\prime}(x) \cup S^{\prime\prime}(x)$ and $S^{\prime}(x) \cap S^{\prime\prime}(x) = \phi.$ In particular, we get that $|S{^\prime}(x)|=|S(x)|-|S^{\prime\prime}(x)|.$ Thus it suffices to consider the two quantities $|S(x)|$ and $|S^{\prime\prime}(x)|$. \\
	\indent (i) Note that $S(x)=\{p \leq x~|~p~\textrm{splits completely in $L$}\}$ by Theorem \ref{prime rep thm1}. Then it follows from Chebotarev density theorem that
	\begin{equation}
	\frac{|S(x)|}{|\{p \leq x~|~p \in P\}|} \rightarrow \frac{1}{|\textrm{Gal}(L/\Q)|}= \frac{1}{2 \cdot h(-8g-4)} ~\textrm{as}~x \rightarrow \infty.
	\end{equation}
	\indent (ii) Similarly, we note that $S^{\prime\prime}(x)=\{p \leq x~|~p~\textrm{splits completely in $L(\zeta_{2g+1})$}\}$ by Lemma \ref{asym lem}. Then, again, by Chebotarev density theorem, we get that
	\begin{equation}
	\frac{|S^{\prime\prime}(x)|}{|\{p \leq x~|~p \in P\}|} \rightarrow \frac{1}{|\textrm{Gal}(L(\zeta_{2g+1})/\Q)|}=\frac{1}{2 \cdot h(-8g-4) \cdot g} ~\textrm{as}~x \rightarrow \infty.
	\end{equation}
	Hence from (4.1), (4.2), and the argument given above, we can conclude that
	\begin{equation*}
	f_g (x)=\frac{|S^{\prime}(x)|}{|\{p \leq x~|~p \in P\}|} \rightarrow \frac{1}{2 \cdot h(-8g-4)}- \frac{1}{2 \cdot h(-8g-4) \cdot g}=\frac{1}{2 \cdot h(-8g-4)} \cdot \left(1-\frac{1}{g} \right)
	\end{equation*}
	as $x \rightarrow \infty.$ \\
	\indent This completes the proof.
\end{proof}

\begin{remark}
	We can obtain a similar result for $g=3$, and we are excluding the case just because we focused on the case of $g \geq 5$ throughout the paper.
\end{remark}
As an immediate consequence of Theorem \ref{asym thm}, we obtain:
\begin{corollary}\label{asym cor}
	For a fixed $G=C_{4g+2}$ (for a Sophie Germain prime $g \geq 5$), we have 
	\begin{equation*}
	\liminf_{x \rightarrow \infty} f_G(x) \geq \frac{1}{2 \cdot h(-8g-4)} \cdot \left(1-\frac{1}{g} \right).
	\end{equation*}
\end{corollary}
It might be interesting to see whether the equality holds or not, depending on $g.$ \\

We conclude this paper by exhibiting one concrete example:

\begin{example}
	As an illustrating example, we provide a table which consists of triples $(x, f_{11}(x), \frac{5}{33} - f_{11}(x))$ for some values of $x$ (up to $10^6$). Note that the numerical values $\frac{5}{33} - f_{11}(x)$ are evaluated by rounding from $9$ decimal places.
	\begin{center}
		\begin{tabular}{c c c}
			\hline
			$x$ & $f_{11}(x)$ & $\frac{5}{33} - f_{11}(x)$\\
			\hline
			$100$ & $\frac{1}{25}$ & $0.11151515$ \\
			\hline
			$150$ & $\frac{2}{35}$ & $0.09437229$ \\
			\hline
			$200$ & $\frac{2}{23}$ & $0.06455863$ \\
			\hline
			$10^3$ & $\frac{11}{84}$ & $0.02056277$  \\
			\hline
			$10^4$ & $\frac{175}{1229}$ &  $0.00912296$  \\
			\hline
			$10^5$ & $\frac{713}{4796}$ & $0.00284960$\\
			\hline
			$10^6$ & $\frac{3949}{26166}$ & $0.00059411$\\
			\hline
		\end{tabular}
		\vskip 4pt
		\textnormal{Table 4} \\
		\textnormal{Triples $(x, f_{11}(x), \frac{5}{33} - f_{11}(x))$.}
	\end{center}
	In fact, Theorem \ref{asym thm} tells us that the value $\frac{5}{33}-f_{11}(x)$ tends to $0$ as $x \rightarrow \infty.$
\end{example}
Also, to visualize the situation by using a graph, we introduce the following algorithm:
\begin{algorithm} [!hbt]\label{Alg_2}
	\caption{~Provide a plotted $2$-dimensional graph $(x,f_g (x))$ for prime numbers $x$.}
	\begin{algorithmic}[1]
		\STATE Choose a Sophie Germain prime $g \geq 5$, upper bound $u$, $i = 1$ 
		\STATE $np =$ the number of primes not exceeding $u$
		\STATE Initialize List $= [], npg = 0$.
		\STATE \textbf{while} $i \leq np$ \textbf{do}
		\STATE ~~\quad $p = (i$-th prime number$)$
		\STATE ~~\quad $s = 1$
		\STATE ~~\quad \textbf{while} $s < \sqrt{\frac{p}{2g+1}}$ \textbf{do}
		\STATE ~~\quad\quad \textbf{if} $p - (2g+1)s^2$ is square \textbf{do}
		\STATE ~~\quad\quad\quad \textbf{if} $p \not= 2g+1$ and $p \not\equiv 1\pmod{2g+1}$ \textbf{do}
		\STATE ~~\quad\quad\quad\quad $npg++$
		\STATE ~~\quad\quad\quad\quad Go to line $16$
		\STATE ~~\quad\quad\quad \textbf{end if}
		\STATE ~~\quad\quad \textbf{end if}
		\STATE ~~\quad\quad $s = s + 1$
		\STATE ~~\quad \textbf{end while}
		\STATE ~~\quad Append $(p,\frac{npg}{i})$ to List.
		\STATE ~~\quad $i = i + 1$
		\STATE \textbf{end while}
		\STATE Plot the points $(p,\frac{npg}{i})$ in List and $y = \frac{5}{33}$ for $x \in [0, u]$ 
	\end{algorithmic}
\end{algorithm}

\newpage
Using Algorithm $2$ for $g = 11$ and $u = 10^{6}$, we obtain the following graph. 
\begin{figure}[h] 
	
	\begin{center}
		
		\includegraphics[width=0.75\linewidth]{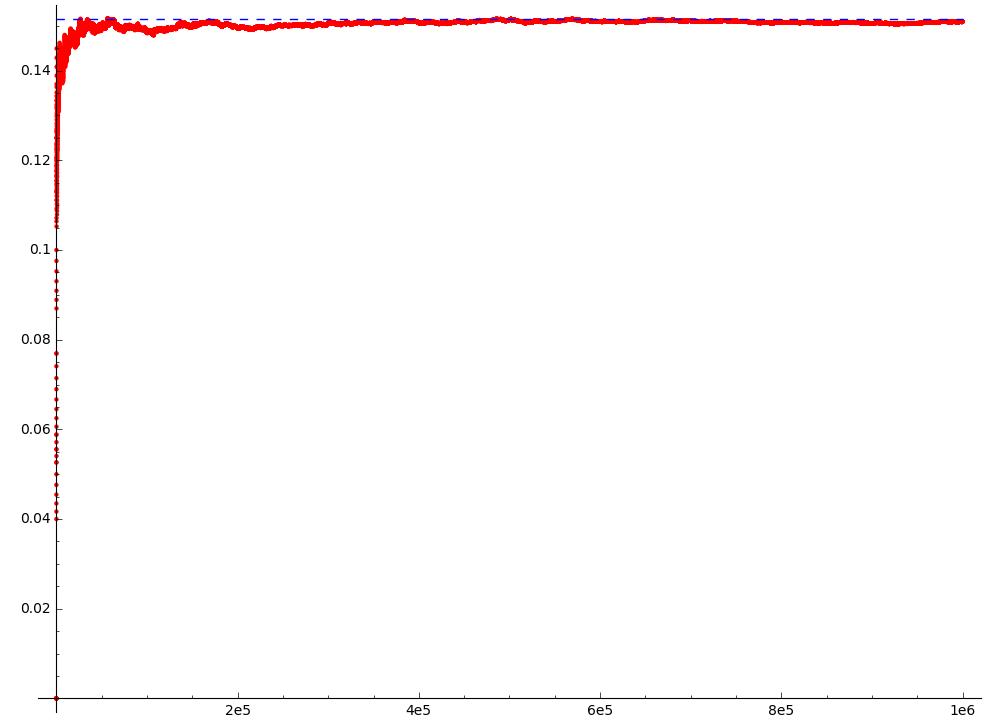}
		
	\end{center}
	
	\caption{A plotted $2$-dimensional graph for $2 \leq x \leq 10^{6}$ with $g = 11$. The red points are $(x,f_{11}(x))$ for prime numbers $x$, and the blue dotted straight line is $y = \frac{5}{33}$.}
	
	\label{fig:long}
	
	\label{fig:onecol}
	
\end{figure}

\section*{Acknowledgement}
WonTae Hwang was supported by a KIAS Individual Grant (MG069901) at Korea Institute for Advanced Study. Kyunghwan Song was supported by Basic Science Research Program through the National Research Foundation of Korea funded by the Ministry of Education (Grant Number: NRF-2019R1A6A1A11051177).

\bibliographystyle{amsplain}

\end{document}